\documentclass[thmsa,12pt]{article}
\normalfont\sffamily
\usepackage{amsfonts}
\usepackage{amssymb}
\usepackage{color}
\usepackage{amsmath}
\usepackage{mathtools}
\usepackage{amsthm}
\usepackage{a4wide}
\usepackage{url}
\usepackage[margin=3cm]{geometry} 

\newtheorem{theorem}{Theorem}
\newtheorem{lemma}{Lemma}
\newtheorem{example}{Example}

\newtheorem{proposition}{Proposition}

\newtheorem{definition}{Definition}

\def \be {\begin{equation}}
\def \ee {\end{equation}}

\makeatletter
\def\@seccntformat#1{\@ifundefined{#1@cntformat}%
   {\csname the#1\endcsname\quad}
   {\csname #1@cntformat\endcsname}}
\newcommand\section@cntformat{}     
\makeatother

\begin{document}



\title{A note on simple games with two equivalence classes of players}
\author{Sascha Kurz$^1$ and Dani Samaniego$^2$\\ \footnotesize $^1$University of Bayreuth, \footnotesize sascha.kurz@uni-bayreuth.de
\vspace{-0.3cm}
\\  \footnotesize $^2$Universitat Politècnica de Catalunya,  \footnotesize daniel.samaniego.vidal@upc.edu}
\date{}
\maketitle
\begin{abstract}
  Many real-world voting systems consist of voters that occur in just two different types. Indeed, each voting system with a {\lq\lq}House{\rq\rq} and a 
  {\lq\lq}Senat{\rq\rq} is of that type. Here we present structural characterizations and explicit enumeration formulas for these so-called bipartite 
  simple games.
  
  \medskip
  
  \noindent
  \textbf{Keywords:}  Boolean functions, Dedekind numbers, voting theory, simple games.
\end{abstract}

\subsection{Introduction}
\label{sec_intro}
Consider voting systems where each player or voter can either agree or disagree to a given proposal. The resulting group decision then is to either accept or to 
dismiss the proposal. The common mathematical model is that of a simple game, which is a monotone Boolean function, see Definition~\ref{ex_simple_game}. Simple games 
where all players are homogeneous have a rather simple structure and are studied in \cite{may1952set}. Also the cases where the players come in just two different types, 
called bipartite simple games, are quite common in real-world voting systems, see e.g.~\cite{felsenthal1998bicameral,taylor1999simple}.  Indeed, each voting system with a 
{\lq\lq}House{\rq\rq} and a {\lq\lq}Senat{\rq\rq} is of that type. 
In general those games are not weighted as they are in the homogeneous case. However, rather weak additional assumptions are sufficient to imply weightedness \cite{herranz2011any}, 
see also \cite{freixas2017characterization}. Weighted games with two types of voters admit a unique minimum integer representation \cite{freixas2014minimum}. For the characterization 
and enumeration of so-called complete simple games with two types of voters we refer to e.g.\ \cite{kurz2013dedekind} and the references mentioned therein, see also 
\cite{freixas2013golden,freixas2014enumeration} for variations. The first subclasses of bipartite simple games were recently enumerated in \cite{freixas2021enumeration}. 
Here we complete the analysis by resolving the open cases and conjectures, see Section~\ref{sec_enumeration}. Especially, we give an explicit formula for the number of 
non-isomorphic simple games with $n$ players and two equivalence classes of players. In \cite{freixas2021enumeration} simple games with two equivalence classes of players 
were parameterized using a matrix notation based on the corresponding minimal winning vectors. In Section~\ref{sec_parameterization} we extend this result to general simple games, 
i.e., simple games with $t\ge 1$ equivalence classes of players. First we have to introduce some notation in Section~\ref{sec_preliminaries}.

\subsection{Preliminaries}
\label{sec_preliminaries}

Let $N=\left\{ 1,2,...,n\right\}$ be a finite set of voters or players. Any subset $S$ of $N$ is called a coalition and the set of all coalitions of $N$ is denoted by
the power set $2^{N}=\{S\mid S\subseteq N\}$. 

\begin{definition}
  \label{def_simple_game}
  A \emph{simple game} is a mapping $v\colon 2^N\to\{0,1\}$ that satisfies $v(\emptyset)=0$, $v(N)=1$, and $v(S)\le v(T)$ for all $\emptyset\subseteq S\subseteq T\subseteq N$, where 
  the finite set $N$ is called the \emph{player set} or \emph{set of players}.
\end{definition}
Let $v$ be a simple game with player set $N$. A subset $S\subseteq N$ is called \emph{winning coalition} if $v(S)=1$ and \emph{losing coalition} otherwise. A winning coalition 
$S\subseteq N$ is called \emph{minimal winning coalition} if all proper subsets $T\subsetneq S$ of $S$ are losing. Similarly, a losing coalition $S$ is called maximal 
losing coalition if all proper supersets $T\supsetneq S$ of $S$ are winning. For an extensive introduction to simple games we refer to \cite{taylor1999simple}.

\begin{example}
  \label{ex_simple_game}
  For player set $N=\{1,2,3\}$ let $v$ be the simple game defined by $v(S)=1$ iff $w(S):=\sum_{i\in S}w_i\ge 3$ and $v(S)=0$ otherwise for all 
  $S\subseteq N$, where $w_1=3$, $w_2=2$, and $w_3=1$.
\end{example}

The winning coalitions of the simple game from Example~\ref{ex_simple_game} are given by $\{1\}$, $\{2,3\}$, $\{1,2\}$, $\{1,3\}$, and $\{1,2,3\}$. Only $\{1\}$ and $\{2,3\}$ are 
minimal winning coalitions.

\begin{definition}
  Let $v$ be a simple game with player set $N$. Two players $i,j\in N$ are called \emph{equivalent} if $v(S\cup\{i\})=v(S\cup\{j\})$ for all $\emptyset\subseteq S\subseteq N\backslash\{i,j\}$.
\end{definition}
In the simple game $v$ from Example~\ref{ex_simple_game} the players $2$ and $3$ are equivalent while player~$1$ is neither equivalent to player~$2$ nor to player~$3$. In general being equivalent  
is an equivalence relation, i.e., $N$ is partitioned into $t\ge 1$ equivalence classes $N_1,\dots,N_t$ such that all pairs of players in such an equivalence class $N_i$ are equivalent while 
two players from two different equivalence classes are not equivalent. In our example we have the equivalence classes $N_1=\{1\}$ and $N_2=\{2,3\}$. Note that the numbering of the 
equivalence classes is arbitrary while their number $t$ is not. We remark that simple games with just one equivalence class, i.e.\ $t=1$, have a pretty simple structure: For a given number 
$n$ of players there exists an integer $1\le q\le n$ such that each coalition is winning iff it has cardinality at least $q$. So, there are exactly $n$ simple games with $n$ players and 
$t=1$. One aim of this paper is to deduce an exact formula for the number of simple games with $n$ players and $t=2$ equivalence classes of players.

We remark that for a given set of players, each simple game $v$ is uniquely characterized by either the set of winning coalitions, the set of losing coalitions, the set of 
minimal winning coalitions, or the set of maximal losing coalitions. While the number of minimal winning coalitions is at most at large as the number of winning coalitions, it 
can be as large as ${n\choose \left\lfloor n/2\right\rfloor}$; attained by the simple game with $n$ players and $t=1$ that is uniquely described by $q=\left\lfloor n/2\right\rfloor$. 
So, our aim is to find a more compact representation for the set of minimal winning coalitions.

\begin{definition}
  Let $v$ be a simple game with player set $N:=\{1,\dots,n\}$ and $N_1,\dots,N_t$ be a partition of $N$ into equivalence classes of players. We set 
  $\overline{n}:=\left(|N_1|,\dots,|N_t|\right)\in\mathbb{N}^t$ and for each coalition $S\subseteq N$ we define the vector $m^S:=\left(|S\cap N_1|,\dots,|S\cap N_t|\right)\in\mathbb{N}^t$. 
\end{definition}
While for each coalition $S\subseteq N$ there is a unique vector $m^S$, there can be several coalitions $S'$ with $m^S=m^{S'}$. However, in this case we have that $S$ is a winning coalition 
iff $S'$ is a winning coalition. Similar statements hold for losing, minimal winning, and maximal losing coalitions. So, we may speak of winning vectors etcetera. To this end we define 
a partial ordering on $\mathbb{N}^t$:
\begin{definition}
  Let $x=(x_1,\dots,x_t)\in\mathbb{N}^t$ and $y=(y_1,\dots,y_t)\in\mathbb{N}^t$. We write $x\preceq y$ or $y\succeq x$ iff $x_i\le y_i$ for all $1\le i\le t$. We abbreviate the 
  cases when $x\preceq y$ and $x\neq y$ by $x\prec y$. Similarly we write $y\succ x$ iff $y\succeq x$ and $x\neq y$. The case when neither $x\preceq y$ nor $x\succeq y$ holds is 
  denoted by $x\bowtie y$ and we say that the two vectors are \emph{incomparable}.    
\end{definition}
By $\mathbf{0}$ we denote the all-zero vector whenever the number of entries, i.e., zeroes, is clear from the context.
\begin{definition}
  Let $v$ be a simple game with player set $N:=\{1,\dots,n\}$ and $N_1,\dots,N_t$ be a partition of $N$ into equivalence classes of players. 
  Let $m\in\mathbb{N}^t$ be a vector with $\mathbf{0}\preceq m\preceq \overline{n}$ and $S\subseteq N$ be an arbitrary coalition with $m=m^S$. We say that 
  \begin{itemize}
    \item $m$ is a \emph{winning vector} iff $S$ is a winning coalition;  
    \item $m$ is a \emph{losing vector} iff $S$ is a losing coalition;
    \item $m$ is a \emph{minimal winning vector} iff $S$ is a minimal winning coalition; and
    \item $m$ is a \emph{maximal losing vector} iff $S$ is a maximal losing coalition. 
  \end{itemize}
\end{definition}
In our example (with fixed equivalence classes $N_1=\{1\}$ and $N_2=\{2,3\}$) the minimal winning vectors are given by $(1,0)$ and $(0,2)$, while $(0,1)$ is the unique maximal 
losing vector.

\subsection{A parameterization of simple games with $\mathbf{t}$ equivalence classes}
\label{sec_parameterization}
Our next aim is to uniquely describe each simple game $v$ by the \emph{counting vector} $\overline{n}$ and a list of minimal winning vectors $m^1,\dots,m^r$. First we 
observe $m^i\bowtie m^j$ for all $1\le i,j\le r$ with $i\neq j$, i.e., different minimal winning vectors are incomparable. Indeed, each list of pairwise incomparable 
vectors $m^1,\dots,m^r$ with $\mathbf{0}\preceq m^i\preceq\overline{n}$ for all $1\le i\le r$ defines a simple game $v$. However, the number of equivalence classes of the resulting simple 
game may be strictly smaller than $t$, i.e., the size of the vectors. If we e.g.\ define a simple game by $\overline{n}=(1,1,1)$ and the minimal winning vectors 
$(1,0,0)$ and $(0,1,1)$, then we end up with the simple game from Example~\ref{ex_simple_game}, which has exactly two equivalence classes of players. 

In order to deduce the extra conditions that guarantee that the number of equivalence classes of the resulting simple game indeed equals $t$, we consider 
a simple game $v$ with player set $N$ that is partitioned into subsets $N_1,\dots,N_t$ such that for all $i,i'\in N_j$, where $1\le j\le t$, the players 
$i$ and $i'$ are equivalent. Under which conditions can we join $N_{\tilde{i}}$ and $N_{\tilde{j}}$ for $\tilde{i}\neq \tilde{j}$? Given arbitrary players 
$i'\in N_{\tilde{i}}$ and $j'\in N_{\tilde{j}}$, we can join $N_{\tilde{i}}$ and $N_{\tilde{j}}$ iff player $i'$ is equivalent to player $j'$. Otherwise 
there exists a coalition $S\subseteq N\backslash\{i',j'\}$ such that 
$v(S\cup\{i'\})\neq v(S\cup\{j'\})$. W.l.o.g.\ we assume that $S\cup\{j'\}$ is a winning coalition. Since $S\cup\{i'\}$ then is a losing coalition, player $j'$ 
cannot be a null player, so that there also exists a coalition $S'\subseteq S$ such that $S'\cup\{j'\}$ is a minimal winning coalition while 
$S'\cup\{i'\}\subseteq S\cup\{i'\}$ is a losing coalition. Thus, there exist a minimal winning coalition $\{j'\}\subseteq T\subseteq N\backslash\{i'\}$ such that  
$T\backslash\{j'\}\cup\{i'\}$ is a losing coalition, i.e., it is not contained in any minimal winning coalition, where we eventually have to interchange the roles 
of $j'$ and $i'$. Translating to vector notation directly gives:
\begin{lemma} 
  Let $\overline{n}\in \mathbb{N}_{>0}^t$, $n=\sum\limits_{i=1}^t \overline{n}_i$, $N=\{1,\dots,n\}$, $N_i=\left\{\sum\limits_{j=1}^{i-1} |N_j|+1,\dots,\sum\limits_{j=1}^{i} |N_j|\right\}$ for all 
  $1\le i\le t$ and $m^1,\dots,m^r$ be pairwise incomparable, where $m^i\in\mathbb{N}^t$ and $\mathbf{0}\preceq m^i\preceq \overline{n}$ for all $1\le i\le t$. For each $1\le i\le t$ we 
  denote by $e^i\in\mathbb{N}^t$ the vector that has a one at position $i$ and zeroes at all other coordinates. If for each $1\le i,j\le t$ 
  with $i\neq j$ there exists an index $1\le h\le r$ such that for either $m':=m^h+e_i-e_j$ or $m':=m^h-e_i+e_j$ we have $\mathbf{0}\preceq m'\preceq\overline{n}$ 
  and the vector $m'$ is losing, i.e.\ there exists no index $1\le h'\le r$ with $m^{h'}\succeq m'$, then the simple game $v$ with player set $N$ defined by $v(S)=1$ 
  iff there exists an index $1\le h\le r$ with $m^S\succeq m^h$ for all $S\subseteq N$ 
  \begin{itemize}
    \item has $N_1,\dots,N_t$ as its equivalence classes of players and
    \item the minimal winning vectors of $v$ are given by $m^1,\dots,m^r$.
  \end{itemize}
\end{lemma}  
\begin{example}
  Let $\overline{n}=(4,2)$, $m^1=(3,0)$, and $m^2=(2,1)$, so that $n=6$, $N=\{1,\dots,6\}$, $N_1=\{1,2,3,4\}$, and $N_2=\{5,6\}$. The minimal winning coalitions of the 
  corresponding simple game $v$ are given by $\{1,2,3\}$, $\{1,2,4\}$, $\{1,3,4\}$, $\{2,3,4\}$, $\{1,2,5\}$, $\{1,3,5\}$, $\{1,4,5\}$, $\{2,3,5\}$, $\{2,4,5\}$, $\{3,4,5\}$, 
  $\{1,2,6\}$, $\{1,3,6\}$, $\{1,4,6\}$, $\{2,3,6\}$, $\{2,4,6\}$, and $\{3,4,6\}$, so that $v$ indeed consists of the $t=2$ equivalence classes of players $N_1$ and $N_2$.\footnote{Note 
  that the simple game $v$ is complete, cf.~\cite{isbell1958class}, so that it does not occur in the list of \cite[Example 3.4.d]{freixas2021enumeration}.} 
\end{example}  
Of course, every relabeling of the vectors $m^1,\dots,m^r$ yields the same simple game, so that we will assume that they are lexicographically ordered.
\begin{definition}
  Let $x=(x_1,\dots,x_t)\in\mathbb{N}^t$ and $y=(y_1,\dots,y_t)\in\mathbb{N}^t$. We write $x\le y$ or $y\ge x$ iff there exist an index $0\le j\le t$ such that 
  $x_i=y_i$ for all $1\le i\le j$ and $x_{j+1}<y_{j+1}$ (if $j<n$). In words we say that $x$ is \emph{lexicographically at most as large} as $y$. We abbreviate the cases when 
  $x\le y$ and $x\neq y$ by $x<y$. Similarly we write $y>x$ if $y\ge x$ and $y\neq x$. Here the relation is called \emph{lexicographically smaller} or 
  \emph{lexicographically larger}, respectively.  
\end{definition}

So, we describe each simple game with $t\ge 1$ equivalence classes of players by a counting vector $\overline{n}\in\mathbb{N}_{>0}^t$ and a matrix $\mathcal{M}\in\mathbb{N}^{r\times t}$ 
consisting of $r$ row vectors $m^i$ satisfying
\begin{enumerate}
  \item[(I)] $\mathbf{0}\preceq m^i\preceq \overline{n}$ for all $1\le i\le r$;
  \item[(II)] $m^i\bowtie m^j$ for all $1\le i<j\le r$; 
  \item[(III)] $m^1>\dots>m^r$; and
  \item[(IV)] for each $1\le i,j\le t$ with $i\neq j$ there exists an index $1\le h\le r$ such that for $m':=m^h+e_i-e_j$ or $m':=m^h-e_i+e_j$ we have 
             $\mathbf{0}\preceq m'\preceq \overline{n}$ and $m^{h'}\not\succeq m'$ for all $1\le h'\le r$.
\end{enumerate} 

Note that while the conditions (I)-(IV) already factor out several symmetries for simple games, interchanging entire equivalence classes of players is not yet considered. E.g.\ 
for $\overline{n}=(4,2)$, $\mathcal{M}=\begin{pmatrix}3&0\\2&1\end{pmatrix}$, $\overline{n}'=(2,4)$, and $\mathcal{M}'=\begin{pmatrix}1&2\\0&3\end{pmatrix}$ the pairs 
$(\overline{n},\mathcal{M})$ and $(\overline{n}',\mathcal{M}')$ satisfy all requirements while representating isomorphic simple games. 

In general we obtain for each representation $(\overline{n},\mathcal{M})$ of a simple game $v$ with $t$ equivalence classes of players and each permutation 
$\pi$ of $\{1,\dots,t\}$ another representation $\left(\overline{n}^\pi,\mathcal{M}^\pi\right)$ of $v$, where 
$$
  \overline{n}^\pi= \left(\overline{n}_{\pi(1)},\dots,\overline{n}_{\pi(t)}\right)
$$  
and $\mathcal{M}^\pi$ consists of the row vectors $\widehat{m}^1,\dots,\widehat{m}^r$ sorted in decreasing lexicographical order, where 
$$
  \widehat{m}^i=\left(m^i_{\pi(1)},\dots,m^i_{\pi(t)}\right)
$$
for all $1\le i\le r$. E.g.\, in our above example we have $\overline{n}'=\overline{n}^\pi$ and $\mathcal{M}'=\mathcal{M}^\pi$ for the permutation $\pi$ 
interchanging $1$ and $2$, which is the only permutation for $t=2$ which is not the identity. Note that the list of representations $\left(\overline{n}^\pi,\mathcal{M}^\pi\right)$, 
that satisfy conditions (I)-(IV), is exhaustive if we start from an arbitrary representation $(\overline{n},\mathcal{M})$ and consider all $t!$ permutations $\pi$ of $\{1,\dots,t\}$.

So, in order to obtain a unique representative for an isomorphism class of simple games with respect to relabeling the players, we have to distinguish one of these 
$\left(\overline{n}^\pi,\mathcal{M}^\pi\right)$. Again we can utilize some kind of lexicographical ordering.
\begin{definition}
  Let $X=\left(x_{i,j}\right)\in\mathbb{N}^{r\times t}$ and $Y=\left(y_{i,j}\right)\in\mathbb{N}^{r\times t}$. We write $X\le Y$ or $Y\ge X$ iff $\widehat{x}\le \widehat{y}$ (or 
  $\widehat{y}\ge \widehat{x}$), where
  $$
    \widehat{x}=\left(x_{1,1},\dots,x_{r,1},x_{1,2},\dots,x_{r,2},\dots,x_{1,t},\dots,x_{r,t}\right)\in\mathbb{N}^{rt}
  $$ 
  and
  $$
    \widehat{y}=\left(y_{1,1},\dots,y_{r,1},y_{1,2},\dots,y_{r,2},\dots,y_{1,t},\dots,y_{r,t}\right)\in\mathbb{N}^{rt}.
  $$ 
  In words we say that $X$ is \emph{lexicographically at most as large} as $Y$. We abbreviate the cases when $X\le Y$ and $X\neq Y$ by $X<Y$. Similarly we write $Y>X$ if 
  $Y\ge X$ and $Y\neq X$. Here the relation is called \emph{lexicographically smaller} or \emph{lexicographically larger}, respectively.  
\end{definition}
So, we might choose the representation $\left(\overline{n}^\pi,\mathcal{M}^\pi\right)$ as {\lq\lq}the{\rq\rq} representative where $\mathcal{M}^\pi$ is lexicographically 
largest. Having the algorithmic complexity in mind, we instead assume that the entries of $\overline{n}$ are weakly decreasing and only consider those permutations $\pi$ 
of $\{1,\dots,t\}$ that fix $\overline{n}$ for the determination of the lexicographical maximum:
\begin{theorem}
  The isomorphism classes of simple games with $n\ge 1$ players and $t\ge 1$ equivalence classes of players are in one-to-one correspondence to pairs 
  $\left(\overline{n},\mathcal{M}\right)$, where $\overline{n}\in\mathbb{N}_{>0}^t$ and $\mathcal{M}\in\mathbb{N}^{r\times t}$ with row vectors $m^1,\dots,m^r$, for 
  some integer $r\ge 1$, satisfying
  \begin{enumerate}
    \item[(a)] $\overline{n}_1\ge\overline{n}_2\ge\dots\ge\overline{n}_t>0$, $\sum_{i=1}^t \overline{n}_i=n$;
    \item[(b)] $\mathbf{0}\preceq m^i\preceq \overline{n}$ for all $1\le i\le r$;
    \item[(c.1)] $m^i\bowtie m^j$ for all $1\le i<j\le r$; 
    \item[(c.2)] $m^1>\dots>m^r$; 
    \item[(d)] for each $1\le i,j\le t$ with $i\neq j$ there exists an index $1\le h\le r$ such that for $m':=m^h+e_i-e_j$ or $m':=m^h-e_i+e_j$ we have 
               $\mathbf{0}\preceq m'\preceq \overline{n}$ and $m^{h'}\not\succeq m'$ for all $1\le h'\le r$; and
    \item[(e)] $\mathcal{M}\ge \mathcal{M}^\pi$ for every permutation $\pi$ of $\{1,\dots,t\}$ with $\overline{n}=\overline{n}^\pi$.           
  \end{enumerate}
\end{theorem}
For the special case $t=2$ the conditions (a)-(e) can be simplified or made more explicit at the very least:
\begin{enumerate}
    \item[(a')] $\overline{n}_1\ge\overline{n}_2>0$, $\overline{n}_1+\overline{n}_2=n$;
    \item[(b')] $\mathbf{0}\preceq m^i\preceq \overline{n}$ for all $1\le i\le r$ (or $0\le m^i_j\le\overline{n}_j$ for all $1\le i\le r$, $1\le j\le t$);
    \item[(c')] $m^i_1>m^{i+1}_1$ and $m^i_2<m^{i+1}_2$ for all $1\le i\le r-1$; 
    \item[(d')] there exists an index $1\le h\le r$ such that for $m':=m^h+(-1,1)$ or $m':=m^h+(1,-1)$ we have $\mathbf{0}\preceq m'\preceq \overline{n}$ and 
               $m^{h'}\not\succeq m'$ for all $1\le h'\le r$; and
    \item[(e')] $\left(m^1_1,\dots,m^r_1\right)\ge\left(m^r_2,\dots,m^1_2\right)$ if $\overline{n}_1=\overline{n}_2$.           
\end{enumerate}  
Only the conversions from (c.1), (c.2) to (c') and from (e) to (e') need a little discussion. If $m^i_1=m^j_1$ for some $1\le i,j\le r$ with $i\neq j$, then we cannot have 
$m^i\bowtie m^j$, which is requested in (c.1). Thus, (c.2) implies $m^i_1>m^{i+1}_1$ for all $1\le i\le r-1$, which is the first part of (c'). The second part of (c') is then 
implied by using $m^i\bowtie m^{i+1}$. It can be easily checked that (c') implies (c.1) and (c.2). For condition (e) we remark that the unique permutation $\pi$ that is not 
the identity interchanges $1$ and $2$, so that $\overline{n}=\overline{n}^\pi$ is only possibly if $\overline{n}_1=\overline{n}_2$. Moreover we have
$$
  \mathcal{M}^\pi=\begin{pmatrix}m^r_2 & m^r_1\\ \vdots & \vdots\\ m^1_2 & m^1_1\end{pmatrix}, 
$$  
so that $\mathcal{M}\ge \mathcal{M}^\pi$ is equivalent to $\left(m^1_1,\dots,m^r_1\right)\ge\left(m^r_2,\dots,m^1_2\right)$.


Similarly, the conditions (I)-(IV) can be rephrased to
\begin{enumerate}
    \item[(I')] $\overline{n}_1\ge\overline{n}_2>0$, $\overline{n}_1+\overline{n}_2=n$;
    \item[(II')] $\mathbf{0}\preceq m^i\preceq \overline{n}$ for all $1\le i\le r$ (or $0\le m^i_j\le\overline{n}_j$ for all $1\le i\le r$, $1\le j\le t$);
    \item[(III')] $m^i_1>m^{i+1}_1$ and $m^i_2<m^{i+1}_2$ for all $1\le i\le r-1$; and
    \item[(IV')] there exists an index $1\le h\le r$ such that for $m':=m^h+(-1,1)$ or $m':=m^h+(1,-1)$ we have $\mathbf{0}\preceq m'\preceq \overline{n}$ and 
               $m^{h'}\not\succeq m'$ for all $1\le h'\le r$.           
\end{enumerate}  
for the special case $t=2$.

\subsection{Enumeration results}
\label{sec_enumeration}
In \cite[Theorem 4]{kurz2013dedekind} the number of complete simple games with $n$ players and two equivalence classes of players was determined using generating 
functions.\footnote{The formula was also proven using more direct adhoc methods.} The parameterization of complete simple games with $t$ equivalence classes from 
\cite{carreras1996complete} and the reformulation of the conditions in terms of integer points in a polyhedron, see \cite[Lemma 1]{kurz2013dedekind}, were the essential 
steps for this approach. Since we have provided a parameterization in Section~\ref{sec_parameterization} or can use the formulation for $t=2$ and non-complete simple games 
in \cite{freixas2021enumeration}, going along the same lines is feasible. For technical reasons we will start to enumerate the pairs $(\overline{n},\mathcal{M})$ satisfying 
conditions (I')-(IV') first, before we apply these results to those pairs $(\overline{n},\mathcal{M})$ that satisfy the conditions (a')-(e').  
  
\begin{lemma}
  \label{lemma_counting_representation}
  Each simple game with $t=2$ equivalence classes of players and $r\ge 2$ minimal winning vectors given by $\overline{n}=\left(\overline{n}_1,\overline{n}_2\right)\in\mathbb{N}_{>0}^2$ and
  $$
    \mathcal{M}=\begin{pmatrix}m^1\\\vdots\\ m^r\end{pmatrix}
  $$
  satisfying the conditions (I')-(III') can be written as
  \begin{equation}
    \overline{n}=\begin{pmatrix}
      z_1+r-1+\sum\limits_{j=1}^{r} x_j &
      z_2+r-1+\sum\limits_{j=1}^{r} y_j
    \end{pmatrix}  
  \end{equation}
  and 
  \begin{equation}
    \mathcal{M}=\begin{pmatrix}
    r-1+\sum\limits_{j=1}^r x_j   & 0 +y_1 \\ 
    r-2+\sum\limits_{j=1}^{r-1} x_j   & 1 +y_1+y_2\\ 
    \vdots & \vdots \\ 
    r-i+\sum\limits_{j=1}^{r-i+1} x_j & i-1 +\sum\limits_{j=1}^{i} y_j \\     
    \vdots & \vdots \\
    1+x_1+x_2  & r-2+\sum\limits_{j=1}^{r-1} y_j \\
    0+x_1     & r-1+\sum\limits_{j=1}^{r} y_j
    \end{pmatrix}
  \end{equation}
  where $x_1,\dots,x_r,y_1,\dots,y_r,z_1,z_2$ are non-negative integers fulfilling
  \begin{equation}
    \label{eq_distr}
    \sum_{i=1}^r x_i\,+\,\sum_{i=1}^r y_i\,+\, z_1\,+\, z_2 \,=\, n+2-2r.
  \end{equation} 
\end{lemma}  
\begin{proof}
  For one direction, we only have to check the conditions (I')-(III'). For the other direction, we state that one can recursively determine the $x_h$, $y_i$, and $z_j$ 
  via 
  \begin{eqnarray*}
    x_1 &=&m^r_1\\
    x_h &=& m^{r-h+1}_1-m^{r-h+2}_1-1\quad\text{for } h=2,\dots,r \\
    y_1 &=& m^1_2  \\ 
    y_i &=& m^i_2-m^{i-1}_2-1\quad\text{for }i=2,\dots,r\\ 
    z_1 &=& \overline{n}_1-(r-1)-\sum_{j=1}^r x_j,\text{ and}\\
    z_2 &=& \overline{n}_2-(r-1)-\sum_{j=1}^r y_j. 
  \end{eqnarray*}  
  Verifying $x_h,y_i,z_j\ge 0$ finishes the proof.
\end{proof}  

Directly from Equation~(\ref{eq_distr}) and the non-negativity of the $x$-, $y$-, and $z$-variables we conclude $2\le r\le\left\lfloor n/2\right\rfloor+1$ (and $n\ge 2$).  
The number of non-negative integer solutions of Equation~(\ref{eq_distr}) is given by 
$$
  {{(n+2-2r)+(2r+2)-1}\choose{(2r+2)-1}}={{n+3}\choose{2r+1}},
$$ 
so that the total number of cases is given by
\begin{equation}
  \label{eq_count_1}
  \sum_{r=2}^{\left\lfloor n/2\right\rfloor+1} {{n+3}\choose{2r+1}} =2^{n+2}-{{n+3}\choose 1}-{{n+3}\choose 3}.
\end{equation}
Next we consider the cases where condition (IV') is violated. These cases are characterized by
\begin{itemize}
  \item $x_2=\dots=x_r=0$;
  \item $y_2=\dots=y_r=0$;
  \item $y_1=0\,\vee\,z_1=0$; and
  \item $x_1=0\,\vee\,z_2=0$, i.e., there are 
\end{itemize} 
$$
4+\sum_{i=1}^{n+2-2r}4=4+4(n+2-2r)
$$
such cases for each $2\le r\le\left\lfloor n/2\right\rfloor+1$.\footnote{For $a=n+2-2r$ we have the four cases $\left(x_1,y_1,z_1,z_2\right)\in\big\{(a,0,0,0),(0,a,0,0),(0,0,a,0), 
(0,0,0,a)\big\}$ and for each $1\le i\le a-1$ we have the four cases $\left(x_1,y_1,z_1,z_2\right)\in\big\{(i,0,a-i,0),(i,0,0,a-i),(0,i,a-i,0), 
(0,i,0,a-i)\big\}$.} Additionally for the case $r=\lfloor\frac{n}{2}\rfloor+1$ and $n$ is even, it remains to add another term,\footnote{With n even, for  $a=n+2-2(\lfloor\frac{n}{2}\rfloor+1)=0$ we have the case $(x_1,y_1,z_1,z_2)=(0,0,0,0)$.} so that the total number of cases is given by:
    $$  \left\{
               \begin{array}{ll}
                 4(n-\left\lfloor\frac{n}{2}\right\rfloor-1)\left\lfloor\frac{n}{2}\right\rfloor & \; \hbox{if n is odd} \\ \\
                 4(n-\left\lfloor\frac{n}{2}\right\rfloor-1)\left\lfloor\frac{n}{2}\right\rfloor+1 & \; \hbox{if n is even} 
               \end{array}
             \right.
$$
which is exactly
\begin{equation}
     \label{eq_count_2}
     (n-1)^2
\end{equation}

for each $n\ge 2$.

The case $r=1$ is treated separately:
\begin{lemma}
  \label{lemma_case_r_1}
  For $t=2$, $r=1$, and each $n\ge 1$ the number of pairs $\left(\overline{n},\mathcal{M}\right)$ satisfying conditions (I')-(IV') is given by 
  \begin{equation}
    \label{eq_count_3}
    \frac{n^3 + 6n^2 - 13n + 6}{6}.
  \end{equation}
\end{lemma}
\begin{proof}
  We write $\overline{n}=\begin{pmatrix}n_1&n_2\end{pmatrix}$ and $\mathcal{M}=\begin{pmatrix}a&b\end{pmatrix}$. The conditions (I')-(IV') are satisfied 
  if 
  \begin{itemize}
    \item $1\le n_1\le n-1$, so that $1\le n_2\le n-1$ for $n_2=n-n_1$;
    \item $0\le a\le n_1$;
    \item $0\le b\le n_2=n-n_1$;
    \item $(a,b)\neq (n_1,n_2)$ and $(a,b)\neq (0,0)$.
  \end{itemize}
  Thus, there are
  $$
    \sum_{n_1=1}^{n-1} \left(\sum_{a=0}^{n_1}\sum_{b=0}^{n-n_1} 1\,-\,2\right)
    = \sum_{n_1=1}^{n-1} \big((n_1+1)\cdot(n-n_1+1)-2\big)=\frac{n^3 + 6n^2 - 13n + 6}{6}
  $$    
  cases.
\end{proof}
Adding (\ref{eq_count_3}) to the right hand side of (\ref{eq_count_1}) and subtracting the right hand side of (\ref{eq_count_2}) yields:
\begin{proposition}
  \label{prop_num_bipartite_simple_games}
  For each $n\ge 2$ the number of pairs $\left(\overline{n},\mathcal{M}\right)$ satisfying conditions (I')-(IV') is given by
  \begin{equation}
    \label{eq_count_4}
     2^{n+2}-n^2-3n-4.
  \end{equation}   
\end{proposition}
The enumeration formula in Proposition~\ref{prop_num_bipartite_simple_games} is only an auxiliary result and our actual aim is a corresponding enumeration formula 
for the number of pairs  $\left(\overline{n},\mathcal{M}\right)$ satisfying conditions (a')-(e'). To this end we have a look at condition (e') again and repeat our 
observation that for $t=2$ the unique permutation $\pi$ of $\{1,2\}$ that is not the identity interchanges $1$ and $2$. Given an arbitrary  pair 
$\left(\overline{n},\mathcal{M}\right)$ we can have
\begin{enumerate}
  \item[(i)] $\overline{n}_1\ge \overline{n_2}$, $\overline{n}_1^\pi\le \overline{n_2}^\pi$, and $\mathcal{M}>\mathcal{M}^\pi$; 
  \item[(ii)] $\overline{n}_1\le \overline{n_2}$, $\overline{n}_1^\pi\ge \overline{n_2}^\pi$, and $\mathcal{M}<\mathcal{M}^\pi$; and 
  \item[(iii)] $\overline{n}_1=\overline{n_2}$, $\overline{n}_1^\pi= \overline{n_2}^\pi$, $\overline{n}=\overline{n}^\pi$, and $\mathcal{M}=\mathcal{M}^\pi$.
\end{enumerate}   
Proposition~\ref{prop_num_bipartite_simple_games} counts the cases falling in categories (i)-(iii) while we actually only want to count the cases falling in category (i) or 
(iii). In order to be more precise, let us denote the corresponding counts by $c_{\text{i}}$, $c_{\text{ii}}$, and $c_{\text{iii}}$, respectively. Since 
$\left(\overline{n}^\pi\right)^\pi$ and $\left(\mathcal{M}^\pi\right)^\pi$ we have $c_{\text{i}}=c_{\text{ii}}$, so that 
\begin{equation}
  \label{eq_burnside}
  c_{\text{i}}+c_{\text{iii}}=\frac{2c_{\text{i}}+2c_{\text{iii}}}{2}=\frac{\Big(c_{\text{i}}+c_{\text{ii}}+c_{\text{iii}}\Big) \,\,+\,\, c_{\text{iii}}}{2},
\end{equation}
i.e., we need a counting formula for $c_{\text{iii}}$.\footnote{While our derivation of Equation~(\ref{eq_burnside}) is rather adhoc and elementary, we remark 
that for the general case $t\ge 2$ we can apply Burnside's lemma, which is sometimes also called Burnside's counting theorem, the Cauchy-Frobenius lemma, or 
orbit-counting theorem.\label{fn_burnside}}

\begin{lemma}
  \label{lemma_symmetric_bipartite_games}
  For each $n\ge 2$ we have $c_{\text{iii}}=0$ if $n$ is odd and 
  \begin{equation}
    \label{eq_count_5}
    c_{\text{iii}}=2^{m+1}-2m-2  
  \end{equation}
  if $n$ is even, where $m=n/2$.
\end{lemma}
\begin{proof}
  We count the number of pairs $\left(\overline{n},\mathcal{M}\right)$ falling in category (iii). First we note that 
  $\overline{n}_1=\overline{n}_2$ implies that $n$ is even, so that we assume that $n$ is even in the following. We will go 
  along the same lines as in the derivation of the enumeration formula in Proposition~\ref{prop_num_bipartite_simple_games}. 
  Since $\mathcal{M}^\pi=\mathcal{M}$, where $\pi$ is the permutation swapping $1$ and $2$, we have 
  $\left(m^1_1,\dots,m^r_1\right)=\left(m^r_2,\dots,m^1_2\right)$. In the context of Lemma~\ref{lemma_counting_representation} 
  this is equivalent to $x_i=y_i$ for all $1\le i\le r$ and $z_1=z_2$. Equation~(\ref{eq_distr}) then simplifies to 
  $$
    2\sum_{i=1}^r x_i\,+\, 2z_1\,=\, n+2-2r,
  $$
  which is equivalent to 
  $$
    \sum_{i=1}^r x_i\,+\, z_1\,=\, m+1-r,  
  $$
  so that we have ${{m+1}\choose r}$ non-negative integer solutions for each $2\le r\le m+1$ and 
  \begin{equation}
    \label{eq_s_1}
    \sum_{r=2}^{m+1} {{m+1}\choose r} =2^{m+1} -{{m+1}\choose 1} -{{m+1}\choose 0}=2^{m+1}-m-2
  \end{equation} 
  solutions in total. The number of cases where condition (IV') is violated is given by $2$ for each $2\le r\le m$ and by $1$ for $r=m+1$, so that 
  the total number of cases is given by
  \begin{equation}
    \label{eq_s_2}
    1+\sum_{r=2}^m 2=2(m-1)+1=2m-1.
  \end{equation}
  For $r=1$ we proceed as in the proof of Lemma~\ref{lemma_case_r_1}. Here we have $n_1=n_2=m$ and $a=b$, so that the number of cases is given by
  \begin{equation}
    \label{eq_s_3}
    \sum_{a=1}^{m-1} 1=m-1.
  \end{equation} 
  Subtracting the right hand side of (\ref{eq_s_2}) from the right hand side of (\ref{eq_s_1}) and adding the right hand side of (\ref{eq_s_3}) yields 
  the stated formula.  
\end{proof}

\begin{theorem}
  For each $n\ge 2$ the number of simple games with $n$ players and two equivalence classes is given by 
  \begin{equation}
     \left\{
               \begin{array}{ll}
                 2^{n+1}-\frac{n^2+3n+4}{2} & \; \hbox{if n is odd} \\ \\
                 2^{n+1}+2^{\frac{n}{2}}-\frac{n^2+4n+6}{2} & \; \hbox{if n is even} 
               \end{array}
             \right.
  \end{equation}
\end{theorem}
\begin{proof}
  As observed, the corresponding number equals the number of pairs $\left(\overline{n},\mathcal{M}\right)$ satisfying conditions (a')-(e'). So, 
  plugging in the formulas of Proposition~\ref{prop_num_bipartite_simple_games} and Lemma~\ref{lemma_symmetric_bipartite_games} into Equation~(\ref{eq_burnside}) 
  yields the stated result. 
\end{proof}
So, indeed the number of bipartite simple games with $n$ players is in $\theta(2^n)$. More precisely, 
the number number of bipartite simple games with $n$ players asymptotically equals $2^{n+1}$.

Of course we may also extract explicit formulas for the number of bipartite simple games with $n$ players and $r$ minimal winning vectors from our intermediate results. 
Finally, we remark that for each pair of fixed parameters $p$ and $r$ it is possible to describe the pairs $(\overline{n},\mathcal{M})$ that satisfy the conditions (I)-(IV) 
as integer points in a suitable polyhedron. As demonstrated in \cite[Section 3.2]{kurz2013dedekind} for complete simple games, we can then apply 
an algorithmic version of Ehrhart theory, see e.g.\ \cite{lepelley2008ehrhart}, where software packages like e.g.\ \texttt{Barvinok} are available, to explicitly compute a 
quasi-polynomial for their number. Being 
interested in the number of simple games with $n$ players, $t$ equivalence classes, and $r$ minimal winning vectors, in terms of $n$, we have to consider the conditions 
(a)-(e) instead of the conditions (I)-(IV). As mentioned in Footnote~\ref{fn_burnside}, we can apply Burnside's lemma to this end and reduce the problem to $t!$ subproblems 
that can be treated as described before. To sum up, the computation of explicit formulas for the number of simple games with $t$ equivalence and $r$ minimal winning vectors 
in terms of the number of players $n$ is algorithmically possible but rather messy. Since we do not expect any {\lq\lq}nice{\rq\rq} formulas we abstain from going into the 
details. Maybe there are more clever ways to at least determine the order of magnitude. However, as far as we know, even the maximum possible number $r$ of minimal winning 
vectors given $t>2$ equivalence classes of players is unknown. 


\end{document}